\documentclass{amsart}

\let\tmp\oddsidemargin
\let\oddsidemargin\evensidemargin
\let\evensidemargin\tmp
\reversemarginpar

\usepackage{amsmath}
\usepackage{array}
\usepackage{arydshln}

\usepackage{float}
\usepackage{bbm}
\usepackage{bm}

\usepackage[utf8]{inputenc} 

\usepackage{amsthm}

\usepackage{imakeidx} 
\usepackage{verbatim}
\usepackage{graphicx} 
\usepackage{enumerate} 
\usepackage[usenames, dvipsnames]{color} 
\usepackage{hyperref}
\hypersetup{
    colorlinks,
    citecolor=OrangeRed,
    filecolor=black,
    linkcolor=blue,
    urlcolor=black}

\newtheorem{theorem}{Theorem}[section]
\newtheorem{proposition}[theorem]{Proposition}

\newtheorem{corollary}[theorem]{Corollary} 

\theoremstyle{definition}
\newtheorem{problem}[theorem]{Problem}
\newtheorem{definition}[theorem]{Definition}

\newtheorem{remarkx}[theorem]{Remark}
\newenvironment{remark}
	{\pushQED{\qed}\remarkx}
	{\popQED\\}{\endremarkx}

\newcommand{\tagarray}{%
\mbox{}\refstepcounter{equation}%
$(\theequation)$%
}

\newcommand\ObsDynOp{A}

\newcommand\ObsDynSp{X} 
\newcommand\ObsMap{\Psi} 
\newcommand\ObsAd[1]{#1^*}

\newcommand\C{\mathbb{C}}

\definecolor{red}{RGB}{255,0,0}

\newcommand{\calA}{\mathcal A}         
\newcommand{\calB}{\mathcal B}

\newcommand{\calE}{\mathcal E}         
         
\newcommand{\calG}{\mathcal G}

\newcommand{\calS}{\mathcal S}         
\newcommand{\calT}{\mathcal T}         
\newcommand{\calU}{\mathcal U}

\newcommand{\R}{\mathbb{R}}
\newcommand{\Z}{\mathbb{Z}}
\newcommand{\N}{\mathbb{N}}

\newcommand{\G}{\mathcal{G}}

\newcommand {\Ker} {\operatorname{Ker}}
\newcommand{\Real}{\operatorname{Re}} 

\usepackage{makeidx}
\makeindex

\title{Dynamical Sampling: a view from control theory}
\author{Rocío Díaz Martín, Ivan Medri, Ursula Molter}

\begin{document}
\maketitle
\begin{center}
\section*{Abstract}
{In this contribution we establish a dictionary between terms in two different areas in order to show that many of the topics studied are common ones - just with a different terminology. We further analyze the relations between the discrete-time and continuous-time versions of the problem, using results from both of these  fields. We  will also differentiate between a discretization of the continuous-time dynamical system and a discrete-time dynamical system itself.}    
\end{center}

\section{Introduction}

In the so-called {\em Dynamical Sampling problem} a signal 
that is evolving in time through an operator $A$ is spatially sub-sampled at multiple times and one seeks to reconstruct it 
from these spatial-temporal samples, thereby exploiting time evolution (see, e.g., \cite{rdm-aadp,rdm-acmt,rdm-accmp,rdm-adk,rdm-ap,rdm-lv,rdm-rclv}).
In some of the literature there is a mention to a possible connection between Dynamical Sampling and {\em Control Theory}. However, there is so far no explicit formulation of this fact.

In this  article we aim to construct bridges between these two areas, since they are two different ways of looking at the exact same phenomena. What one community calls dynamical sampling the other one calls observability. Both fields have been thoroughly explored and it is our intention to compare what has been proven in each of them,  in order to provide  a kind of {\em dictionary} which allows to translate the results from one field to the other.
We hope that establishing the equivalences will allow both communities to benefit from each other. For example, whereas Dynamical Sampling has mostly focused on discrete-time systems, Observability has mostly focused on continuous-time. 

This brings us to the second purpose of this paper. We analyze relations between discrete-time and continuous-time dynamical systems using results from both fields and our dictionary. It is in fact true that the relation is completely understood when dealing with finite dimensional spaces \cite{rdm-cd2,rdm-hk}. For infinite dimensions, the topic is much more subtle. Detailed proofs and further results will also be exhibited in the forthcoming paper \cite{rdm-DDM20}. 

We organize this article as follows:
In Section \ref{rdm-section 1} we recall the definitions from Control Theory and Dynamical Sampling, highlighting all the instances in which they coincide and provide the link between them. 
Specifically, in Subsection \ref{rdm-control theory}  we include the background on Control Theory where the notion of Observability comes from, in Subsection \ref{rdm-dynamical sampling} we do the same but for the case of Dynamical Sampling, and in Subsection \ref{rdm-connections} we will use the previous definitions to relate these theories.  

In Section \ref{rdm-section 2} we analyze some results from both, Control Theory and Dynamical Sampling. 
We will consider mostly infinite dimensional spaces. 
Some general theorems can be proved independently of the time being finite/infinite or the dynamics being continuous/discrete.
For more specialized results we then separate the theory in finite and infinite sampling times: Subsections \ref{rdm-Finite_Continuous_Time} and  \ref{rdm-sec_inf-time}, respectively.

\section{Control theory and Dynamical sampling: Overview and relations}\label{rdm-section 1}

We first summarize without proof the basic facts about controllability and observability. We will be concerned mainly about the latter. However, since they are dual notions (as will be clear from Theorem \ref{rdm-connection_control_obs}) it is useful to know how to translate from one to the other when studying the observability problem. 

We then establish the classical problem in dynamical sampling and its generalizations. 

As a conclusion one can see that Dynamical Sampling can be considered a particular case of Observability in certain settings. 

\subsection{Control theory notions: discrete and continuous systems}\label{rdm-control theory}\

Lets us begin with some terminology of control theory of linear dynamical systems (see, e.g., \cite{rdm-Coron, rdm-Niko, rdm-Tucsnak}). For the purpose of this exposition we will not work under the most general hypothesis. In particular, we will restrict the definitions to Hilbert spaces and continuous operators. For the more general setting, i.e. Banach spaces instead of Hilbert spaces and densely defined close operators instead of bounded operators, we refer the interested reader to the above cited references. 
We will always consider both, the discrete- and the continuous-time systems.

Let $X$, $Y$ and $\mathcal{U}$ be complex separable Hilbert spaces. 
We call $X$ the \textit{state space}, $Y$ the \textit{observation space} and $\mathcal{U}$ the \textit{control space}.
Let $A:X\to X$, $B:X\to Y$ and $C:\mathcal{U}\to X$ be linear and bounded mappings. They are called \textit{dynamic operator},  \textit{observation operator} and  \textit{control operator}, respectively.

The operator $A$ is the infinitesimal generator of the \textit{uniformly continuous} one-parameter group $\{e^{tA}\}_{t\in \R}$,
where $e^{tA}:=\sum_{n=0}^{\infty}\frac{(tA)^n}{n!}$ and the convergence is in operator norm (cf. \cite{rdm-Pazy}). 
We say that $A$ is \textit{exponentially stable} if, for all $t\geq 0$, $\|e^{tA}\|\leq e^{t\omega}$ for some $\omega<0$. 
This is equivalent to have $\Real(\lambda) < 0$ for all $\lambda \in \sigma(A)$, where $\sigma(A)$ denotes the spectrum of $A$. When the parameter $ \, t \, $ is discrete, we say that $A$ is \textit{strongly stable} if $A^kx\to 0$ when $k\to \infty$ for each $x\in X$. This is equivalent to the condition $|\lambda|<1$ for all $\lambda \in \sigma(A)$.

We will work with \textit{continuous linear systems} of the form
\begin{equation}\label{rdm-sist din}
        \begin{cases}
            \dot{x}(t)=Ax(t)+Cu(t), \qquad x(0) =x_0, \qquad t\in[0,\tau),\\
            y(t) = Bx(t),
        \end{cases}
\end{equation}
where $\tau$ could be finite or $\tau=\infty$. Consider the Hilbert space of square-integrable vector-valued functions
\begin{equation*}
    L^2([0,\tau),\calU):=\left\{u:[0,\tau)\to \calU : \, \int_0^\tau \|u(t)\|^2_\calU \,  dt<\infty \right\}.
\end{equation*}
For any \textit{initial state} $x_0\in X$ and $u\in L^2([0,\tau),\calU)$, the \textit{general mild solution} (cf. \cite{rdm-Pazy}) to \eqref{rdm-sist din} is given by
\begin{equation}\label{rdm-solution}
    x(t)=e^{tA}x_0+\int_0^t e^{(t-s)A}Cu(s) ds.
\end{equation}
For each finite time $0\leq t\leq \tau$, the  operators $\Theta(t):L^2([0,\tau),\calU)\to X$, given by $\Theta(t)u:=\int_0^t e^{(t-s)A}Cu(s) ds$, are called \textit{controllability maps} (cf. \cite{rdm-Tucsnak}).

Analogously, we will consider  \textit{discrete linear systems} of the form
\begin{equation}\label{rdm-sist din discrete}
        \begin{cases}
            x(k+1)=Ax(k)+Cu(k), \qquad x(0) =x_0,\\
            y(k) = Bx(k),
        \end{cases}
    \end{equation}    
where $k\in\Z_+:=\N\cup \{0\}$. 
We denote the Hilbert space of square-summable vector-valued functions by
\begin{equation*}
    \ell^2(\Z_+,\calU):=\left\{u:\Z_+\to \calU : \, \sum_{k=0}^\infty \|u(k)\|^2_\calU   <\infty \right\}.
\end{equation*}
(The time domain $\Z_+$ could be chosen to be finite $\{0,1,...,\gamma\}$ for some $\gamma\in\Z_+$.)
For $x_0\in X$ and $u\in \ell^2(\Z_{+},\calU)$, the general solution of \eqref{rdm-sist din discrete} is given by
\begin{equation*}
    x(k+1)=A^{k+1}x_0+\sum_{j=0}^kA^{k-j}Cu(j).
\end{equation*}
In this case we define the controllability maps $\Theta(k):\ell^2(\Z_+,\calU)\to X$ (for each $k\in\Z_+$) as  $\Theta(k)u:=\sum_{j=0}^kA^{k-j}Cu(j)$.

One may say that \textit{to observe} the system \eqref{rdm-sist din} (resp. \eqref{rdm-sist din discrete}) means to recover the initial data $x_0$ by  knowing only the output function $t\mapsto y(t)$ (resp. $k \mapsto y(k)$). On another side, 
\textit{to control}  \eqref{rdm-sist din} or \eqref{rdm-sist din discrete} means to act on the system with a suitable function $u$ in such a way that starting from a given initial state $x_0$ the systems attain at a time $\tau$ the given final state $x(\tau)=x_1$. 
Since there are several notions of observability and controllability, in what follows we define some of them and  put special emphasis on the duality between observation and controllability.

\subsubsection{Controllability}\

For infinite-dimensional systems (that is when the state space $X$ is of infinite dimension) there are at least two important controllability notions:  \textit{exact} and \textit{approximate} 
controllable systems.

For simplicity we consider only finite time $\tau<\infty$ ($\gamma<\infty$). 
At first we restrict our attention to continuous systems of the form \eqref{rdm-sist din}. 

The \textit{set of reachable states
from $x_0\in X$ at finite time $\tau$} is defined as
\begin{equation*}
     \calA(\tau)x_0:=\{x(\tau): \, \text{there exist } u\in L^2([0,\tau), \calU) \text{ such that } \eqref{rdm-sist din} \text{ admits the solution } x\}.
\end{equation*}

\begin{definition}
The pair $(A,C)$ is called
      \begin{enumerate}
          \item \textit{Exactly controllable at finite time $\tau$} if  for every $x_0\in X$, $\calA(\tau)x_0=X$;
          \item \textit{Approximately controllable at finite time $\tau$} if for every $x_0\in X$, ${\calA(\tau)x_0}$ is dense in $X$.
      \end{enumerate}
      We denote them briefly by ECO and ACO  respectively.
  \end{definition}

Note that for these definitions the operator $B$ in system \eqref{rdm-sist din} is irrelevant.

     \begin{remark}\label{rdm-coro_123_niko} (\cite[Corollary 1.2.3]{rdm-Niko}) Using \eqref{rdm-solution} we have that the set of reachable states is the affine space
    \begin{eqnarray*}
        \calA(\tau)x_0=e^{\tau A}x_0+
        \left\{\Theta(\tau)u : \, u\in L^2([0,\tau),\calU)\right\}.
    \end{eqnarray*}
   Also, for arbitrary $x_0,x_0' \in X$, we have $\calA(\tau)x_0=X$ if and only if $\calA(\tau)x_0'=X$.
      Respectively,  ${\calA(\tau)x_0}$ is dense in $X$ if and only if ${\calA(\tau)x_0'}$ in dense in $X$ for all $x\in X$.
In particular, setting $x_0=0$, we have that  the pair $(A,C)$ is exactly controllable (respectively, approximately controllable)  at moment $\tau$ if and only if 
$ \, \Theta(\tau)$ is sujective (respectively, $\Theta(\tau)$ has dense image in $X$). 
  \end{remark}

For the discrete case \eqref{rdm-sist din discrete}
we have exactly the same definitions of controllability changing $\tau\in [0,\infty)$ to $\gamma\in \Z_+$ and replacing $L^2([0,\tau), \calU)$ by $\ell^2(\{0,1,...,\gamma\},\calU)$.

\subsubsection{Observability}\

For the case of observability we will look at the systems \eqref{rdm-sist din} with the control term $C\equiv 0$.

As in the case of controllability, for infinite-dimensional systems there are at least two important observability notions, each depending on time. These are called \textit{exact} and \textit{approximate} 
observable.
In this section we introduce these notions and explore how they are related to each other for linear systems of the form \eqref{rdm-sist din} or \eqref{rdm-sist din discrete}. 

We will state the definitions for the case of continuous systems. The necessary changes for the discrete case are given at the end of the section.

    Since we are interested in the output function of the system \eqref{rdm-sist din}, we introduce the \textit{output operator until time $\tau$}\index{Output operator until time $\tau$} or \textit{observability map} $\ObsMap_\tau:X\to L^2([0,\tau),Y)$ 
    given by 
    \begin{equation*}\index{$\ObsMap_\tau$}
        (\ObsMap_\tau x_0) (t):=
            Be^{tA}x_0 \qquad \forall x_0\in X.
    \end{equation*}

     Before the main definitions of the section, an \textit{admissibility} condition for the operator $B$ must be given (see, e.g., \cite{rdm-Tucsnak}).    
    
    \begin{definition}\index{Admissible observation operator}\index{Observation operator!admissible}(Admissible observation operator) An operator $B\in \mathcal{B}(X,Y)$ is called an \textit{admissible observation operator at time $\tau$} for $\{e^{tA}\}$ if $\ObsMap_\tau$ is continuous. 
    \end{definition}
    
    The notion of admissibility can also be extended for operators $B$ which do not belong $\calB(X,Y)$ (see, e.g., \cite[Sections 2.10 and 4.3]{rdm-Tucsnak}).

    \begin{remark}\label{rdm-admissi_finite_time}\
        \begin{enumerate}
         \item $B$ is admissible for $\{e^{tA}\}$ at time $\tau>0$ if and only if there exists a constant $K_\tau\geq 0$ such that 
        \begin{equation*}\label{rdm-local 1}
            \int_0^\tau \|Be^{tA} x_0\|^2_Y dt \leq K_\tau^2\|x_0\|_X^2 \qquad \forall x_0 \in X.
        \end{equation*}
            \item Since $B:X\to Y$ is a bounded linear operator,  $\Psi_\tau$ is continuous for every finite time $0<\tau<\infty$.
            Indeed, 
            \begin{equation*}
                \int_0^\tau \|Be^{tA} x\|^2 dt\leq \frac{e^{2\|A\| \tau}-1}{2\|A\|}\|B\|^2  \|x\|^2.
            \end{equation*}

            \item If $\{e^{tA}\}$ is exponentially stable, then $\|\Psi_\tau\|\leq K$ for all $0\leq\tau<\infty$. Therefore, in this case,   $\Psi_\tau$ is continuous for $\tau=\infty$ and we have
            \begin{equation*}
                \int_0^\infty \|Be^{tA} x\|^2 dt\leq -\frac{1}{2\omega} \|B\|^2   \|x\|^2,
            \end{equation*}
            where $\omega<0$ is such that $\|e^{tA}\|\leq e^{\omega t}$ for all $t\geq 0$.
        \end{enumerate}
    \end{remark}
    
    \begin{definition}\label{rdm-def_obs}
        Let $\tau\in (0,\infty)$ or $\tau=\infty$. Let $B$ be an admissible observation operator at time $\tau$. 
        \begin{enumerate}
            \item The  pair $(A,B)$ is called \index{Exactly observable at time $\tau$}\index{Observable! exactly at time tau}\textit{exactly observable at time $\tau$} if $\ObsMap_\tau$ is bounded from below. 
            \item The  pair $(A,B)$ is called \index{Approximately observable}\index{Observable! approximately}\textit{approximately observable at time $\tau$} if $\Ker \ObsMap_\tau = \{0\}$. 
        \end{enumerate}
    To shorten notation we write EOB and AOB 
    for the exactly and approximately 
    observable conditions, respectively. 
    \end{definition}

    \begin{remark}\label{rdm-6.1.2 rdm-Tucsnak}\cite[Remark 6.1.2]{rdm-Tucsnak} 
        Exact observability implies  approximate observability.
    \end{remark}

    These definitions  have some equivalent formulations (see, e.g., \cite[Chapter 6]{rdm-Tucsnak}).
First, note that  $B$ is admissible and $(A,B)$ is exactly observable if and only if  $\Psi_\tau$ is bounded  above and  below. Hence, by the Open Mapping Theorem this is equivalent to say that $\Psi_\tau$ is invertible on its image. In this case one can recover the initial state $x_0$ from the observations $y=\Psi_\tau x_0$ exactly as $x_0=\Psi_\tau^\dagger y$, where $\Psi_\tau^\dagger:=(\Psi_\tau^{*}\Psi_\tau)^{-1}\Psi_\tau^{*}$ is the Moore-Penrose pseudoinverse of $\Psi_\tau$. These observations yield the following proposition.

    \begin{proposition}\label{rdm-equiv_eob}\index{Exactly observable}\index{Observable! exactly} Equivalent definitions of exactly observable at time $\tau$: 
    \begin{enumerate}
        \item The pair $(A,B)$ is EOB at time $\tau$.
        \item Any initial state $x_0\in X$ can be expressed from the corresponding output function $y=\ObsMap_\tau x_0$ via a bounded operator. That is, there exists a bounded operator $E_\tau: L^2([0,\tau),Y)\rightarrow X$ such that $E_\tau(\ObsMap_\tau x_0)= x_0$ for all $x_0$ in X.
        \item The observability Grammian for time $\tau$,  $Q_\tau := \ObsMap_\tau^*\ObsMap_\tau$, is strictly positive. That is $\langle Q_\tau x, x\rangle > 0$, for every $x\neq 0$.
    \end{enumerate}
    \end{proposition}
     
    For approximate observable systems we have that if the observation function $y=\ObsMap_\tau x_0$ is zero, then $x_0 = 0$. This is sometimes called a \textit{unique continuation property}. Even though this does not imply that $\Psi_\tau$ is bounded below (and therefore with bounded inverse), it does guaranty that we have only one possible initial state for every output.

    \begin{proposition}\label{rdm-equiv_aob}\index{Approximately observable}\index{Observable! approximately} Equivalent definitions of approximately observable at time $\tau$: 
    \begin{enumerate}
        \item The pair $(A,B)$ is AOB at time $\tau$.
        \item  The observability Grammian for time $\tau$ satisfies $\Ker Q_\tau = \{0\}$.
    \end{enumerate}
    \end{proposition}

 Finally, for discrete-time dynamical systems we have the same notions of observability replacing  $e^{tA}$ by $A^k$. Specifically, the observability map is given by
\begin{equation*}\index{$\ObsMap_\tau$}
        \begin{array}{c}
        \ObsMap_\gamma:X\longrightarrow \ell^2(\{0,1,...,\gamma\},Y) \\
        (\ObsMap_\gamma x_0) (k):=
            BA^kx_0,
        \end{array}
\end{equation*}
where we can take $\gamma=\infty$ every time $\tau$ is allowed to be infinite.

\subsubsection{Connections between observability and control}\label{rdm-connections_contr_obs}\ 

Let us now compare observability and control. We will consider two dynamical systems: one  determined by $A$ and the other by $A^*$. Since in our setting  $X$ and $Y$  are Hilbert spaces, their duals  can be identified with themselves.
We  summarize both dynamics in the following table.  \\

\begin{table}
\label{rdm-tab:1}       
\begin{tabular}[\textwidth]{p{2.1cm} |c| c}
\hline\noalign{\smallskip}
 & \multicolumn{1}{c}{Discrete time} & \multicolumn{1}{|c}{Continuous time}  \\
\hline 
Observability  problem & 
         $\begin{cases}
            x(k+1)=Ax(k),\quad x(0) =x_0\\ y(k)=Bx(k)
            \end{cases} $\tagarray\label{rdm-do}
          & 
        $
            \begin{cases}
            \dot{x}(t)=Ax(t), \quad x(0) =x_0\\ y(t)=Bx(t)
            \end{cases}
        $\tagarray\label{rdm-co}  \\ \hdashline[2pt/2pt]
       Observability map & 
        $
            \begin{array}{c}
            \ObsMap_\gamma:X\longrightarrow \ell^2(\{0,1,...,\gamma\},Y)\\
            (\ObsMap_\gamma x)(k)=BA^k x
            \end{array}
        $   & 
        $
            \begin{array}{c}
            \ObsMap_\tau:X\longrightarrow L^2([0,\tau),Y)\\
            (\ObsMap_\tau x)(t)=Be^{tA}x 
            \end{array}
        $ \\ \hdashline[2pt/2pt]
Observation & $y=\Psi_\gamma x_0$ & $y=\Psi_\tau x_0$ \\
\noalign{\smallskip}\hline\noalign{\smallskip}
Control problem & 
        $
            \begin{cases}
            x(k+1)=A^*x(k)+B^*u(k)\\ x(0) =x_0
            \end{cases}
        $\tagarray\label{rdm-dc}   & 
        $
            \begin{cases}
            \dot{x}(t)=A^*x(t)+B^*u(t) \\ x(0) =x_0
            \end{cases}
        $ \tagarray\label{rdm-cc}  \\ \hdashline[2pt/2pt]
Controllability maps & 
            $\begin{array}{c}
            \Theta(k):\ell^2(\{0,1,...,\gamma\},Y)\longrightarrow X\\
            \Theta(k)u=\sum_{j=0}^k(A^*)^{k-j}B^*u(j)
            \end{array}
        $  & 
        $
            \begin{array}{c}
            \Theta(t):L^2([0,\tau),Y)\longrightarrow X\\
            \Theta(t)u=\int_0^t e^{(t-s)A^*}B^*u(s)ds\end{array}
            $  \\ \hdashline[2pt/2pt]
Solution & $x(k+1)=(A^*)^{k+1}x_0+\Theta(k)u$ &
        $x(t)=e^{tA^*}x_0+\Theta(t)u$ \\
\noalign{\smallskip}\hline\noalign{\smallskip}
\end{tabular}
\end{table}

\newpage

 In   
\eqref{rdm-do} and \eqref{rdm-co}  the state space is $X$,  the space of observations is $Y$, and $B$ is the observation operator. 

In \eqref{rdm-dc} and \eqref{rdm-cc} the state space is again $X$, $\calU:=Y$  is the space of controls, and $C:=B^*$ is the control operator. 
 
 One can check that
 $\Psi_\tau^*=\Theta(\tau)$, and  $\Psi_\gamma^*=\Theta(\gamma)$. Using this, Propositions \ref{rdm-equiv_eob} and \ref{rdm-equiv_aob}, and Remark \ref{rdm-coro_123_niko} one can conclude the following theorem. 
 
  \begin{theorem}\label{rdm-connection_control_obs} Let $A\in \calB(X)$,  $B\in\calB(X,Y)$ and $\tau<\infty$ (resp. $\gamma<\infty$ for the discrete case).
    \begin{enumerate}
        \item \label{rdm-1} The pair $(A,B)$ is exactly \textbf{observable} at time $\tau$  if and only if the dual pair $(A^*,B^*)$ is exactly \textbf{controllable} at time $\tau$ (resp. $\gamma$).
        \item \label{rdm-2} The pair $(A,B)$ is approximately \textbf{observable} at time $\tau$ (resp. $\gamma$) if and only if the dual pair $(A^*,B^*)$ is approximately \textbf{controllable} at time $\tau$ (resp. $\gamma$).
    \end{enumerate}  
  \end{theorem}
 For a detailed proof of \ref{rdm-1}. see \cite[Theorem 1.3.4]{rdm-Niko} and  for \ref{rdm-2}. see  \cite[Theorem 1.3.3]{rdm-Niko}.
   
In view of the preceding theorem we say that the systems \eqref{rdm-do} and \eqref{rdm-dc} (resp. \eqref{rdm-co} and \eqref{rdm-cc}) are dual to each other.

This duality will be useful later when trying to relate dynamical systems and control theory. 

It is worth noticing though, that for this duality we use strongly that $\tau$ (resp. $\gamma$) is finite.  For infinite time systems the duality is not as simple. 
Hence, in that case we do not 
 define the control problem because working directly  with the observability formulation is much more convenient for our task. 
 

\subsection{The dynamical sampling problem: discrete and continuous}\label{rdm-dynamical sampling}\

The classical sampling and reconstruction problem consists in recovering a scalar-valued function  from the knowledge of its values at certain points of the spatial domain $X$. In the dynamical sampling problem, the set of \textit{space samples} is replaced by a set of \textit{space-time samples}.

The dynamical sampling problem, introduced in \cite{rdm-acmt}, was stated as follows: Let $x:X\to \C$  be a function in a complex separable Hilbert space $(X,\langle\cdot,\cdot\rangle)$, and assume that $x$ evolves through an evolution linear bounded operator $A: X \to X$ so that the function at time $k\in\Z_{+}$ has evolved to become $x^{(k)}:=A^k x$. 
The space $X$ can be identified with $\ell^2(I)$ where $I=\{1,\dots,d\}$  in the finite dimensional  case and  $I =\N$ in the infinite dimensional case. Let  $\{e_i\}_{i\in I}$ denote the standard basis of $\ell^2(I)$.
The \textit{time-space sample} at time $k\in \Z_+$ and location $i\in I$, is the value $$x^{(k)}(i):=\langle A^{k}x,e_i\rangle.$$ 
In this way to each pair $(i,k)\in I\times \Z_+$ associate a
 sample value $x^{(k)}(i)$. Using this notation, the  dynamical sampling problem can then be described as:  
Under which conditions on the operator $A$, and a set $S\subseteq I\times \Z_+$, can  every $f\in \ell^2(I)$ be recovered in a stable way from the samples in $S$, that is from $\{x^{(k)}(i): \, (i,k) \in S\}$.

Without using the identification with $\ell^2(I)$, sampling a function $x\in X$ can be associated to computing its components along a fixed family of vectors $\calG\subset X$. In that case, the dynamical sampling problem is to recover $x$ (in a stable way) from the samples
\begin{equation}\label{rdm-samples}
   \{\langle A^k x , g\rangle\}_{k\in \{0,1,...,\gamma\}, g \in \calG}, 
\end{equation}
where $\gamma\in\Z_+$ or $\gamma=\infty$.  Note that this is not the most general formulation of the problem since the way it is stated implies that the sampling set $\calG$ remains the same for every time $k$. In the case where $X$ is finite dimensional, the  article \cite{rdm-acmt} also considered slightly more general sampling schemes.

From now on let us assume that $\G$ is a countable set (either finite or infinite). For each $g\in \calG$, let $S_g$ be the operator from $ X$ to $ X_g:=\ell^2(\{0,1,...,\gamma\})$, defined by $(S_g x) (k):= \langle A^k x,g\rangle$ and define $\calS$ to be the operator $\calS:=\bigoplus_{g\in \calG}S_g$. 
Then $x$ can be recovered from the samples \eqref{rdm-samples} in a \textit{stable way} if and only if there exist constants $c_1,c_2>0$ such that 
\begin {equation*}
c\|x\|^2_X\le\|\mathcal S x\|^2=\sum\limits_{g\in \calG } \|S_g x\|_{X_g}^2\le c_2\|x\|^2_X.
\end{equation*}
Or equivalently,  
\begin{equation}\label{rdm-frame_cond}
  c_1\|x\|^2\le\sum\limits_{g \in \calG} \sum_{k=0}^\gamma|\langle x,(A^{k})^{\ast} g\rangle|^2 \le c_2\|x\|^2.  
\end{equation}
That is, if the set of vectors
\begin{equation}\label{rdm-disc_sampling_set}
 \{(A^*)^{k} g\}_{t\in\{0,1,...,\gamma\}, g\in \calG}   
\end{equation}
 is a \textit{frame} for $X$ (cf. \cite[Lemma 1.2]{rdm-acmt}). Therefore, the dynamical sampling problem can also be stated as follows.

\begin{problem}\label{rdm-prob_sampl_disc} Let  $X$ be a complex separable Hilbert space and $\calG\subset X$ a countable set of vectors. Let $A$ be a bounded linear operator and $\gamma\in\Z_+ \cup \{\infty\}$. Under which conditions is the set $\{(A^*)^{k} g\}_{t\in\{0,1,...,\gamma\}, g\in \calG}$
 a frame for $X$. 
\end{problem}

From this starting point, in later developments, the dynamical sampling community focused, amidst other things, on generalizing the discrete iterations of the operator $A$ to continuous scenarios. In particular, in \cite{rdm-FIACPO} the authors, working with the notion of \textit{continuous frames}, dealt with the problem of determining when the set
\begin{equation}\label{rdm-cont_samples}
\{(A^*)^{t} g\}_{t\in[0,\tau), g\in \calG}   
\end{equation}
  is a \textit{semi-continuous frame} or at least a \textit{complete} or  \textit{Bessel} system. The interval $[0,\tau)$ is allowed to be finite or infinite.  In the previously mentioned article $(A^*)^{t}$ was defined using functional calculus for $A$ a normal operator.
  
  We remark that under special conditions on the operator $A$, the dynamic given by $A^t$ is equivalent to that of $e^{tA}$ which resembles the continuous dynamical systems studied previously. 

  For simplicity of the exposition we recall the definitions of continuous frame and Bessel system.

\begin{definition}\label{rdm-def_cont_frame} (\textit{Continuous frame}\index{Continuous frame})  
    Let $X$ be a complex Hilbert space and let $(\Omega,\mu)$ be a measure space with positive measure $\mu$. A mapping $F:\Omega \rightarrow X$ is called a continuous frame with respect to $(\Omega, \mu)$, if
    \begin{enumerate}
        \item \label{rdm-weak_med} $F$ is \textit{weakly-measurable}, i.e., $\omega \mapsto \langle x, F(\omega)\rangle$ is a measurable function on $\Omega$ for all $x\in X$;
        \item there exist positive constants $c_1$ and $c_2$ such that 
        \begin{equation}\label{rdm-cont_frame_eq}
            c_1\|x\|^2 \leq \int_\Omega \lvert \langle x, F(\omega)\rangle\rvert^2 \  d\mu(\omega)\leq c_2\|x\|^2,\text{ for all } x \in X.
        \end{equation}
    \end{enumerate}
    Here the largest (smallest) constant $c_1$ ($c_2$) is called lower (upper) continuous frame bounds. In addition, $F$ is called a \textit{tight} continuous frame\index{tight frame} if $c_1=c_2$. The mapping $F$ is called \textit{Bessel}\index{Bessel} if the second equality (\ref{rdm-cont_frame_eq}) holds. In this case, $c_2$ is called a Bessel constant. 
\end{definition}

Notice that if $F$ from Definition \ref{rdm-def_cont_frame} satisfies \ref{rdm-weak_med}. and the Bessel condition,  then the map
    \begin{align*}\label{rdm-analysis_op}
        B:X\longrightarrow L^2(\Omega)\\
      \notag  Bx(\omega):=\langle x, F(\omega)\rangle 
    \end{align*}
    is well-defined, linear and bounded. It is called the \textit{analysis operator} of the family $\{F(\omega)\}_{\omega\in\Omega}$. The adjoint operator $B^*$ is called the \textit{synthesis operator} of the family $\{F(\omega)\}_{\omega\in\Omega}$.
        If $F$ is a continuous frame for $X$, then the so called \textit{frame operator} $B^*B$ is invertible.

Since we work with countable sets $\calG$, we can take 
$\Omega=\G\times[0,\tau)$ (resp. $\Omega=\G\times\{0,1,...,\gamma\}$), and $\mu$ as the product of the counting  measure on $\G$ and the Lebesgue measure on $[0,\tau)$ (resp. the counting measure on $\{0,1,...,\gamma\}$). In this case,   $F$  is called a {\em semi-continuous frame} (resp. frame) and equation \eqref{rdm-cont_frame_eq} applied to  (\ref{rdm-cont_samples}) becomes
\begin{equation}\label{rdm-SCF}
c_1\|x\|^2\leq\sum\limits_{g\in\G}\int_0^{\tau}|\langle  x,(A^\ast)^t g\rangle|^2 dt \leq c_2\|x\|^2,\text{ for all } x\in X
\end{equation}
(resp. \eqref{rdm-cont_frame_eq} applied to  \eqref{rdm-disc_sampling_set} becomes \eqref{rdm-frame_cond}).

    \subsection{Relations between the dynamical sampling and observability problem}\label{rdm-connections}\

We will establish the relations between Dynamical Sampling and Control / Observability. To highlight the statements we set them in box form.

Let $X$ be an arbitrary complex separable Hilbert space and $A\in\calB(X)$. 
We first show the connection between the set of vectors $\calG$ from Problem 
\ref{rdm-prob_sampl_disc} in the Dynamical Sampling setting and the observation operator $B$ from equations \eqref{rdm-do} or \eqref{rdm-co}.
Given a countable set of vectors $\G\subset X$, we consider the observation space $Y=\ell^{2}(|\G|)$ and $B: X \to Y$ the analysis operator of  $\calG$, i.e. $Bx:=\left(\langle x, g\rangle\right)_{g\in\G}$.
If $\calG$ is a Bessel set of vectors in $X$, then $B\in\calB(X,Y)$. 
Conversely, let $Y$ be any complex separable Hilbert space and let $B\in\calB(X,Y)$. If we consider $\{e_i\}_{i\in\N}$ an orthonormal basis for $Y$, then we can write
\begin{equation*}
    Bx=\sum_{i\in\N}\langle Bx, e_i\rangle_Y e_i= \sum_{i\in\N}\langle x, B^* e_i\rangle_X e_i, \qquad \forall x\in X.
\end{equation*}
Taking $\G:=\{B^*e_i\}_{i\in\N}$ one may think of $B$ as an operator from $X$ to $\ell^2(\N)$ of the form $x\mapsto\left(\langle x, B^* e_i\rangle\right)_{i\in\N}$. Since $B$ is bounded, $\G$ is a Bessel set of vectors in $X$. 

\bigskip

\fbox{
  \parbox{\textwidth}{
From now on, we think the observation operators $B$ as the analysis operator of a set of sampling vectors $\calG$.
}}

\bigskip

Let $\gamma\in \Z_+\cup\{\infty\}$. The observability map 
$\Psi_\gamma:X\to \ell^2(\{0,1,...,\gamma\},Y)$ of the system \eqref{rdm-do}  takes the form
\begin{eqnarray*}
(\ObsMap_\gamma x_0) (k)&=&BA^kx_0\\
&=&\left(\langle A^kx_0,g\rangle\right)_{g\in\G}\\
&=&\left(\langle x_0,(A^*)^kg\rangle\right)_{g\in\G}.
\end{eqnarray*}

\bigskip

\fbox{
  \parbox{\textwidth}{
Hence, for discrete time,  the \textbf{observability map} $\Psi_\gamma$ associated to the dynamical system \eqref{rdm-do} is the \textbf{analysis operator} of $\{(A^*)^kg\}_{t\in\{0,1,...,\gamma\},g\in \G}$.}}

\bigskip

We have, for all $x\in X$,
   \begin{equation}\label{rdm-norma_Psi_disc}
    \|\Psi_\gamma x\|_{\ell^2(\{0,1,...,\gamma\},Y)}^2=\sum_{k=0}^\gamma \|BA^k x\|^2 =\sum_{k=0}^\gamma \sum_{g\in \G}|\langle x,(A^*)^k g\rangle|^2.
   \end{equation}
Therefore, Problem \ref{rdm-prob_sampl_disc} is the same as deciding whether the pair $(A,B)$ is exactly observable. This is a consequence of equation \eqref{rdm-norma_Psi_disc}: In Definition \ref{rdm-def_obs} the bounded below and above conditions for the observation map $\Psi_\gamma$ are exactly formula \eqref{rdm-frame_cond}.

\bigskip

\fbox{
  \parbox{\textwidth}{
The pair $(A,B)$ associated to the discrete system  \eqref{rdm-do} is \textbf{exactly observable} at time $\gamma$ if and only if the  family of vectors $\{(A^*)^k g\}_{t\in\{0,1,,,,\gamma\}, g\in \G}$ is a \textbf{frame} for $X$.

\noindent The \textbf{observability Grammian}  $Q_\gamma=\ObsMap_\gamma^*\ObsMap_\gamma$  is the \textbf{frame operator} associated to $\{(A^*)^kg\}_{t\in\{0,1,...,\gamma\},g\in \G}$
}}

\bigskip

Moreover, with the exact same argument, $(A,B)$ will be AOB if and only if the set \eqref{rdm-disc_sampling_set} is complete in $X$. 

\bigskip

\fbox{
  \parbox{\textwidth}{
The pair $(A,B)$ associated to the discrete system  \eqref{rdm-do} is \textbf{approximately observable} at time $\gamma$ if and only if the  family of vectors $\{(A^*)^k g\}_{t\in\{0,1,,,,\gamma\}, g\in \G}$ is a \textbf{complete system} in $X$.
}}

\bigskip

It is worth noticing that even though it is not explicitly stated in the Dynamical Sampling problem that the set $\calG$ needs to be Bessel, it is a necessary condition as can be deduced from the upper bound of equation  \eqref{rdm-frame_cond}. Also, this necessary condition implies that the operator $B$ is admissible (see Remark \ref{rdm-admissi_finite_time}).

\bigskip

\fbox{
  \parbox{\textwidth}{
The operator $B$ is \textbf{admissible} for $\{A^k\}_{k\in\Z_+}$ at time $\gamma$ if and only if $\{(A^*)^kg\}_{t\in\{0,1,...,\gamma\},g\in \G}$ is a \textbf{Bessel system} in $X$.
}}

\bigskip

Analogously, let $\tau\in[0,\infty)\cup \{\infty\}$. 
Under the same specifications on the space $Y$ and the operator $B$, the observability map 
$\Psi_\tau:X\to L^2([0,\tau),Y)$  takes the form $$(\ObsMap_\tau x_0) (t)=\left(\langle x_0,e^{tA^*}g\rangle\right)_{g\in\G}$$ and its norm is given by $\|\Psi_\tau x\|^2=\int_0^\tau \sum_{g\in \G}|\langle x,e^{tA^*}g\rangle|^2 dt$. Thus, the exact (resp. approximate) observability conditions on the pair $(A,B)$ in the continuous case \eqref{rdm-co}, turn to be the problem of deciding whether the set $\{e^{tA^*}g\}_{t\in [0,\tau),g\in \G}$ is a continuous frame (resp. complete system).  

\bigskip

\fbox{
  \parbox{\textwidth}{
    The operator $B$ is \textbf{admissible} for $\{e^{tA}\}_{t\in[0,\infty)}$ at time $\tau$ if and only if $\{e^{tA^*}g\}_{t\in[0,\tau),g\in \G}$ is a \textbf{Bessel system} in $X$.\\
    The observability map $\Psi_\tau$ associated to the dynamical system \eqref{rdm-co} is the \textbf{analysis operator} of $\{(e^{tA^*}g\}_{t\in\{0,1,...,\gamma\},g\in \G}$.}}
    
\bigskip
    
    \fbox{
  \parbox{\textwidth}{
    The pair $(A,B)$ associated to the continuous system  \eqref{rdm-co} is \textbf{approximately observable} at time $\tau$ if and only if the  family of vectors $\{e^{tA^*} g\}_{t\in[0,\tau), g\in \G}$ is a \textbf{complete system} in $X$.
    }}
    
    \bigskip
    
    \fbox{
  \parbox{\textwidth}{
    The pair $(A,B)$ associated to the continuous system  \eqref{rdm-co} is \textbf{exactly observable} at time $\tau$ if and only if the family of vectors $\{e^{tA^*} g\}_{t\in[0,\tau), g\in \G}$ is a \textbf{semi-continuous frame} for $X$.\\
    The \textbf{observability Grammian}   $Q_\tau=\ObsMap_\tau^*\ObsMap_\tau$ is the \textbf{frame operator} associated to  $\{e^{tA^*} g\}_{t\in[0,\tau), g\in \G}$.  
     }}
     
\bigskip

\begin{remark}
    The problem of Dynamical Sampling for continuous time was originally proposed for $\{A^t\}$. When $A$ is a self-adjoint and strictly positive operator, the operators $A^t$ coincide with $e^{t\widetilde{A}}$ for an appropriate choice of  $\widetilde{A}\in\calB(X)$, for all $t\geq 0$. When $A$ is not self-adjoint, it is not clear how to find  (or decide if it exists) an appropriate $\widetilde{A}$. However, most results from dynamical sampling for $A^t$ can be adapted to $e^{tA}$. Since our main purpose in this paper is to explore the connections between Control Theory and Dynamical Sampling, we chose to work with the family of vectors 
    \begin{equation*}\label{rdm-cont_samples_exp}
    \{e^{tA^*} g\}_{t\in[0,\tau), g\in \calG} \qquad \text{ instead of } \qquad \{(A^*)^t g\}_{t\in[0,\tau), g\in \calG},
    \end{equation*}
     adjusting the results from \cite{rdm-FIACPO}.  
\end{remark}

\section{A comparison between results from control theory and dynamical sampling}\label{rdm-section 2}

In finite dimension spaces $X$, conditions for exact and approximate observability and the frame condition, for the discrete and continuous problems, have been found in both areas and are  completely understood. We could mention \cite{rdm-cd1} from the Observability perspective  and \cite{rdm-acmt} from the Dynamical Sampling one, among others. Surprisingly, the discrete and continuous cases, as well as the exact and approximate observability conditions, are all the same (cf. \cite{rdm-cd2}). This case is also simple in another way: For finite dimensional spaces, letting the system evolve infinite time is never necessary and adds no value, so only finite time sampling schemes shall be considered.

In infinite dimensional cases, the relation between the different conditions is much more involved. The exact and approximate observability notions are no longer equivalent, the discrete and continuous case have different necessary and sufficient conditions, and sampling infinitely in time can prove to be an advantage in certain cases. 

From the point of view of Control Theory the analysis of all these aspects is quite complete for dynamic operators $A$ having an unconditional basis of eigenvectors. In general the results are proved for continuous scenarios and the analogous results are folklore for the discrete case. We aim to make them explicit every time it is possible. From the point of view of Dynamical Sampling, more things can be added since they only assume the operator to be normal. Moreover, conditions on the discrete case are always explicitly stated. It can be proved that the relations between the continuous and discrete that are true for finite dimensions do not hold anymore in infinite dimensions.

For the first results we consider with operators $A$  having an unconditional basis of eigenvectors. 

\begin{definition}
    A sequence of vectors $\{\varphi_n\}$ in a Hilbert space $X$ is said to be
    \begin{enumerate}
        \item an \textit{unconditional sequence} if there exist positive constants $c_1,c_2$ such that for every finite sequence $(a_n)$ it holds that
        \begin{equation}\label{rdm-uncondition_cond}
            c_1\sum |a_n|^2\|\varphi_n\|^2\leq \|\sum a_n \varphi_n\|^2 \leq c_2\sum |a_n|^2\|\varphi_n\|^2;
        \end{equation}
        \item a \textit{Riesz sequence} if there exist positive constants $c_1,c_2$ such that for every finite sequence $(a_n)$ it holds that
        \begin{equation}\label{rdm-riesz_cond}
            c_1\sum |a_n|^2\leq \|\sum a_n \varphi_n\|^2 \leq c_2\sum |a_n|^2.
        \end{equation}
    \end{enumerate}
    If $\{\varphi_n\}$ is a complete system in $X$ it is called
    \begin{enumerate}
        \item an \textit{unconditional basis} if it satisfies  \eqref{rdm-uncondition_cond};
        \item a \textit{Riesz basis} if it satisfies \eqref{rdm-riesz_cond}.
    \end{enumerate}
\end{definition}

In \cite{rdm-Niko} a fundamental theorem [Theorem 3.3.2] was proved in terms of controllability. We restate it here using the duality between control and observability that has been discussed in \ref{rdm-connections_contr_obs}. It is written in a general way so that it includes discrete and continuous dynamics as well as finite and infinite time. For that reason, we will denote  by $\Psi$ the observability map, and it will be understood that it takes the form of the discrete $\Psi_\gamma$ or continuous $\Psi_\tau$ depending on the context. In this formulation it provides general conditions on exact observability that can then be specialized for each case. Later, in Proposition  \ref{rdm-prop_eq_normas_eq} we will write these specializations explicitly.

\begin{theorem}\label{rdm-cond_nec_suf}
Let $X$ be a complex separable Hilbert space and $A\in\calB(X)$.
    Let $\{\varphi_n\}_{n\geq 1}$ be the eigenvectors of $A$, 
        $$A\varphi_n = -\lambda_n\varphi_n,\qquad n\geq 1,$$
    and suppose that they constitute an unconditional basis in $X$.
    Consider 
    \begin{equation}\label{rdm-En_cont}
      \calE_n := \Psi\varphi_n \qquad n\geq 1.
    \end{equation}
        The pair $(A,B)$ is exactly observable if and only if 
        \begin{enumerate}
            \item $\{\calE_n\}$ is an unconditional sequence  and
            \item $\|\calE_n\| \asymp \|\varphi_n\|$, that is, there exists positive constants $c_1$ and $c_2$ such that $$c_1\|\varphi_n\|\leq\|\calE_n\| \leq c_2 \|\varphi_n\|.$$
        \end{enumerate}
\end{theorem}	

\begin{proof}
    ($\textit{2}\Rightarrow \textit{1}$)\,: Since  $\{\varphi_n\}_{n\geq 1}$ is an unconditional basis, for each $x\in X$ we have 
    \begin{equation*}
        x = \sum_{n\geq 1} a_n \varphi_n,  \, \text{ with }
        \|x\|^2 = \sum_{n\geq 1} |a_n|^2 \|\varphi_n\|^2<\infty.
    \end{equation*}
    Then, by the continuity and linearity of $\Psi$
    \begin{align*}
        (\Psi x) (t)  = \sum_{n\geq 1} a_n \calE_n(t).
    \end{align*}
    By hypothesis, we get 
    \begin{equation*}
        \|\Psi x\|^2 =\sum_{n\geq 1} |a_n|^2 \|\calE_n\|^2\asymp \sum_{n\geq 1} |a_n|^2 \|\varphi_n\|^2= \|x\|^2.
    \end{equation*}
   That is, the observability map is bounded from above and below, so $(A,B)$ is EOB. \\
   ($\textit{1}\Rightarrow \textit{2}$)\,: The first statement is a consequence of two facts. One is that the exact observability is equivalent to saying that the observation operator $\Psi$ is an isomorphism on its image. The second, is that topological isomorphisms map unconditional basis into unconditional basis. Therefore, $\Psi$ will map the set $\{\varphi_n\}_{n\geq 1}$ into an unconditional basis of the image of $\Psi$, $\Psi(X)$, and in particular into an unconditional sequence. 
    Thus, $\{\calE_n\}$ is an unconditional sequence.\\
    Moreover, since exact observability also means that $\Psi$ is bounded above and below, there exist constants $c_1$, $c_2 >0$ such that
    \begin{equation*}
        c_1 \|\varphi_n\|^2 \leq 
        \|\Psi\varphi_n\|^2 = \|\calE_n\|^2 \leq c_2 \|\varphi_n\|^2.
    \end{equation*}
\end{proof}

\begin{corollary}  Under the notation of Theorem \ref{rdm-cond_nec_suf}, suppose in addition that $\{\varphi_n\}_{n\geq 1}$ is an Riesz basis in $X$.
    The pair $(A,B)$ is exactly observable if and only if 
    $\{\calE_n\}$ is a Riesz sequence.
\end{corollary}	

\begin{proposition}\label{rdm-prop_eq_normas_eq}
 Under the hypothesis of Theorem \ref{rdm-cond_nec_suf}, $\|\calE_n\| \asymp \|\varphi_n\|$ if and only if 
\begin{itemize}
    \item for continuous and finite time $\calT=[0,\tau)$: 
    \begin{eqnarray}
        &\sup_{n\geq 1} |\Real (\lambda_n)|<\infty \label{rdm-autoval_bounded}, \\
    &\inf_{n\geq 1} \frac{\|B\varphi_n\|}{\|\varphi_n\|}>0  \label{rdm-cond_B_cont_inf1}
    \end{eqnarray}
    \item for continuous and infinite time $\calT=[0,\infty)$: there exist constants $c_1,c_2 >0$ such that for all $n\geq 1$
        \begin{eqnarray}
           &\Real (\lambda_n)>0,  \label{rdm-autoval_neg}\\
            &c_1\sqrt{2\Real(\lambda_n)} \leq  \frac{\|B\varphi_n\|}{\|\varphi_n\|}\leq c_2\sqrt{2\Real(\lambda_n)} \label{rdm-cond_B_cont_inf}
        \end{eqnarray} 
    \item for discrete and infinite time $\calT=\Z_+$: there exist constants $c_1,c_2 >0$ such that for all $n\geq 1$
    \begin{eqnarray}
           &|\lambda_n|<1, \label{rdm-autoval_disco}\\
            &c_1\sqrt{1-|\lambda_n|^2} \leq  \frac{\|B\varphi_n\|}{\|\varphi_n\|}\leq c_2\sqrt{1-|\lambda_n|^2}. \label{rdm-cond_B_disc}
        \end{eqnarray}
\end{itemize}
\end{proposition}

Condition \eqref{rdm-autoval_bounded} is trivially satisfied if $A$ a bounded operator. In general this proposition and the last theorem still hold for closed operators with almost the same proof. 

\begin{proof} 
    We start the proof for the continuous finite case. 
    Let us begin assuming  $\|\calE_n\| \asymp \|\varphi_n\|$. That is, there exist $c_1,c_2>0$, independent of $n$, such that
    \begin{align}\label{rdm-in_proof_eq_1}
        c_1 \|\varphi_n\|^2 \leq 
        \|\calE_n\|^2_{L^2([0,\tau),Y)} = \frac{e^{-2\Real(\lambda_n) \tau} - 1}{-2\Real(\lambda_n)} \|B\varphi_n\|^2 \leq c_2 \|\varphi_n\|^2.
    \end{align}
    Using that $B$ is bounded and the left inequality we conclude
    \begin{align*}
        c_1 \leq 
         \frac{e^{-2\Real(\lambda_n) \tau} - 1}{-2\Real(\lambda_n)} \|B\|^2.
    \end{align*}
    Therefore, $\sup_{n\geq 1}\{\Real(\lambda_n)\}< \infty$. Since  $\|e^{tA}\|\leq e^{t\|A\|}$, we also know that $$\sup_{n\geq 1}\{\Real(-\lambda_n)\}< \infty.$$ These proves that $\sup_{n\geq 1}|\Real(\lambda_n)|<\infty$. Using this and the left bound again we obtain,
    \begin{align*}
        c_1 \|\varphi_n\|^2 \leq 
         \frac{e^{-2\Real(\lambda_n) \tau} - 1}{-2\Real(\lambda_n)} \|B\varphi_n\|^2 \leq K \|B\varphi_n\|^2,
    \end{align*}
    where $K$ is a positive constant. Therefore
    \begin{equation*}
       \inf_{n\geq1}\frac{\|B\varphi_n\|}{\|\varphi_n\|}>0. 
    \end{equation*}
    The reverse implication is straightforward noticing that there exist nonzero positive constants $c_1, c_2$ and $d_1, d_2$ such that 
    \begin{align*}
        c_1\|\varphi_n\|\leq \|B\varphi_n\|\leq c_2 \|\varphi_n\|, \qquad \text{and }\\
        d_1 < \frac{e^{-2\Real(\lambda_n) \tau} - 1}{-2\Real(\lambda_n)} <d_2, \qquad \text{for all } n.
    \end{align*}
    Consequently, \eqref{rdm-in_proof_eq_1} holds.

    The second and the third items are easy to prove. Just consider that the analogous of \eqref{rdm-in_proof_eq_1} are
    \begin{equation*}
        c_1 \|\varphi_n\|^2 \leq 
        \|\calE_n\|^2_{L^2([0,\infty),Y)} = \left(\int_{0}^\infty e^{-2\Real(\lambda_n)t}dt\right) \|B\varphi_n\|^2 \leq c_2 \|\varphi_n\|^2,
    \end{equation*}
    and
    \begin{equation*}
        c_1 \|\varphi_n\|^2 \leq 
        \|\calE_n\|^2_{\ell^2(\Z_+,Y)} = \left(\sum_{k=0}^\infty |\lambda_n|^{2k}\right) \|B\varphi_n\|^2 \leq c_2 \|\varphi_n\|^2.
    \end{equation*}
    Then, for the first case,  $e^{-2\Real(\lambda_n)t}\in L^2([0,\infty))$ for all $n\geq 1$ so  $\Real(-\lambda_n)<0$ and $\int_{0}^\infty e^{-2\Real(\lambda_n)t}dt=\frac{1}{2\Real(\lambda_n)}$. This implies  \eqref{rdm-cond_B_cont_inf}. For the second case, the geometric series converges if and only if $|\lambda_n|<1$. Therefore \eqref{rdm-cond_B_disc} must hold. The reverse implications are trivial.
\end{proof}

\begin{remark} Suppose that we are under the hypothesis of Theorem \ref{rdm-cond_nec_suf}. The exact observability implies the following necessary conditions on the point spectrum $\sigma_p(A)$ of the dynamic operator $A$:
    \begin{itemize} 
    \item For continuous and finite time $\calT=[0,\tau)$,  there exist $\alpha>0$ such that
    \begin{equation*}
    \sigma_p(A)\subset\{\lambda\in\C: \, -\alpha<\Real(\lambda)<\alpha\}.    
    \end{equation*}
    \item For continuous and infinite time $\calT=[0,\infty)$, there exist $\alpha>0$ such that
    \begin{equation*}
    \sigma_p(A)\subset\{\lambda\in\C: \, -\alpha<\Real(\lambda)<0\}.    
    \end{equation*} 
    \item  For discrete and infinite time $\calT=\Z_+$, 
    \begin{equation*}
    \sigma_p(A)\subset\{\lambda\in\C: \, |\lambda|<1\}=:\mathbb{D}.    
    \end{equation*} 
    \end{itemize}
\end{remark}

In the remainder of this subsection we will separate the finite-time case from the infinite-time.

\subsection{Finite  time}\label{rdm-Finite_Continuous_Time}\

We will consider, unless other specifications, pairs of the form $(A,B)$ associated to dynamical systems as \eqref{rdm-co} for $0<\tau<\infty$.
We mention that the results from Dynamical Sampling that we will cite are proved in the literature for evolution operators of the form $\{A^t\}_{t\in[0,\tau)}$ instead of $\{e^{tA}\}_{t\in[0,\tau)}$. We will enunciate them using the dynamics given by $\{e^{tA}\}_{t\in[0,\tau)}$ to show the relation with Control Theory. We recall that, in  general, it is not true that $e^{tA} = (e^A)^t$, so the adaptations are not immediate.  As a rule of thumb, most results will follow just by assuming that it is true, the technical details and full proofs of this facts should be considered in each case and are going to be presented in \cite{rdm-DDM20}. 

Some considerations have to be mentioned. Since $e^{tA}$ is always defined for bounded operators (even for non normal ones), the normality hypothesis that is often required in Dynamical Sampling is not further necessary. In a similar fashion, since $e^{tA}$ is always invertible, the invertibility of $A$ will not be required.  

Taking all this into account, we will present most results written as similar to the original ones as possible.

From Theorem \ref{rdm-cond_nec_suf}, we know that exact observability results can be derived from checking whether $\{\calE_n\}$ is an unconditional sequence of the corresponding space. In general, the conditions for that to happen for finite time require results that exceed the aim of this presentation. The interested reader can find results in this direction in \cite{rdm-ai1,rdm-ai2,rdm-Niko3,rdm-Niko2,rdm-Niko}. However,  without appealing to this condition,s still several things can be said for finite time observability.  

In particular, one desires the amount of sampling points (both spatial and temporal) to be as \textit{small} as possible. We compile some theorems that claim something in this regard. 

The first result establishes that every time one can recover the initial state of the system by sampling continuously during a finite time, one can choose to sample intermittently in time without losing information or stability. We could call this process as a discretization in the sampling scheme or observability problem. Notice that this gives no relation between the discrete and continuous problem. It is just a middle ground where one still has a system evolving continuously but samples discretely. 
The following result is almost verbatim of \cite[Theorem 5.4]{rdm-FIACPO} when changing $A^t$ to $e^{tA}$.

\begin{theorem}\label{rdm-ScToDscr}
	Let $X$ be a complex separable Hilbert space, $A\in\calB(X)$ and    $\G$  a Bessel system of vectors in $\ObsDynSp$. The following are equivalent.
	\begin{enumerate}
	    \item $\{e^{t\ObsAd{\ObsDynOp}}g \}_{g\in\G,t\in[0,\tau)}$ is a semi-continuous frame for $X$.
	    \item There exists $\delta>0$ such that for any finite set $T=\{t_i:i=1,\ldots,n\}$ with $0=t_1< t_2<\ldots<t_n\leq t_{n+1}=\tau$ and $|t_{i+1}-t_{i}|<\delta$ for all $i\in\{1,...,n\}$,  the system $\{e^{t\ObsAd{\ObsDynOp}}g\}_{g\in \G,t\in T}$ is a frame for $\ObsDynSp$.
	    \item There exists a finite partition $T=\{t_i:i=1,\ldots,n\}$ of $[0,\tau)$ with $0= t_1< t_2<\ldots<t_n\leq \tau$  such that $\{e^{t\ObsAd{\ObsDynOp}}g \}_{g\in\G,t\in T}$ is a frame for $\ObsDynSp$.
	\end{enumerate}
\end{theorem}

Last theorem and its corollary can be further generalized for non Hilbert spaces in the context of the observability problem. Its proof follows the arguments given in \cite{rdm-FIACPO} but without requiring the inner product of the space $X$. In particular,  one can work with a $B$ that is not the analysis operator of any set of vectors. The proofs of these facts will be presented in \cite{rdm-DDM20}.

As a conclusion, when in presence of systems that are exactly observable in finite time, there is no necessity to sample continuously. As a side product,  one can also conclude that when we have observable systems in finite time, it is necessary to sample at an infinite quantity of spatial points. (Otherwise, the discretized version will be a finite set of vectors and it can never be a frame.) That is obviously a disadvantage that we will try to solve recurring to a trade-off between infinite space and infinite time sampling schemes. The following corollary is an adaptation of \cite[Corollary 5.3]{rdm-FIACPO}.

\begin{corollary}\label{rdm-coro5.3_rdm-FIACPO} Let $X$ be  of infinite dimension Hilbert space and let $A\in \calB(X)$. If $\{e^{tA^*}g\}_{g\in\G,t\in[0,\tau)}$ is a frame for $X$, then $|\G|=+\infty $.
\end{corollary}

Here is another point where both communities have reached similar results. Without explicitly recurring to the discretization of the sampling process, a similar result was also established in \cite[Theorem 1.3.19]{rdm-Niko} for the controllability problem. We rewrite it to be directly readable for the observability problem. Its original version  should be translated using the duality relationship of Theorem \ref{rdm-connection_control_obs}.

\begin{theorem}\label{rdm-Theorem 1.3.19Niko}
Let $X$ be an infinite dimension Hilbert space and let $A\in \calB(X)$. If the observation  operator $B\in\calB(X,Y)$ is compact, then the pair $(A,B)$  cannot be exactly observable at any moment $\tau<\infty$. In particular, if $B$ is of finite rank, then the pair $(A,B)$ cannot be exactly observable at any moment $\tau<\infty$.  
\end{theorem}

Another important question beyond the discretization problem is how much time one has to sample to have observability. In this direction one can analyze the dependence on time of three different properties of systems $(A,B)$. The first one is the admissibility condition on $B$, the second and third are the approximate and exact observability of the pair. 

It can be proved that  the admissibility condition on $B$ is independent of the sampling time. We have the following theorem which is inspired by  \cite[Theorem 4.3]{rdm-FIACPO} but without require that $A$ is normal.

\begin{theorem}\label{rdm-Teo4.3rdm-FIACPO}
 	Let $X$ be a complex separable Hilbert space and $A\in\calB(X)$. A countable set of vectors $\G\subset X$ is a Bessel system if and only if, for any number $\tau$, $\{e^{tA^*}g\}_{g\in\G,t\in[0,\tau)}$ is a Bessel system in $X$.  
	Equivalently, a linear map $B:X\to Y$ is continuous if and only if it is admissible for $\{e^{tA}\}$ at some finite time $\tau$.
\end{theorem}

\begin{proof}
    By Remark \ref{rdm-admissi_finite_time},  we have that if $B\in \calB(X,Y)$, then it is admissible at any finite time.
    
    For the converse, let $\tau_0<\infty$ such that $B$ is admissible for  $\{e^{tA}\}$ at $\tau_0$. If $\tau_0=0$ there is nothing to prove.
    For $\tau_0>0$, it is immediate that $B$ is admissible for $\{e^{tA}\}$ at any time $0\leq \tau\leq\tau_0$. 
    
    For $\tau$ sufficiently small (smaller than $\tau_0$),  consider  $T:X\to X$ defined by,
    \begin{equation*}\label{rdm-Besselell}
    T x:=\int_{0}^{\tau}e^{tA}x \, dt.
    \end{equation*}
    Then, $T\in\mathcal{B}(X)$ since $\tau<\infty$.
    Thus, if we prove that $T$ is invertible, $B$ is bounded if and only if $BT$ is.  
    
    To prove that $T$ is invertible, consider, for $t>0$ and $z\in\C$, the  function $g(t,z):=\frac{e^{t z}-1}{z}=\sum_{k=1}^\infty \frac{t^kz^{k-1}}{k!}$. Notice that $g$ is analytic as a function of $z$, and if $t>0$ is sufficiently small it is non-vanishing for all $z\in \sigma(A)$ (the spectrum of $A$). We claim that $T x=g(\tau,A)x$ for all $x\in X$. In that case, $\sigma(T)=g(\tau,\sigma(A))$, and $0\not\in \sigma(T)$ if $\tau$ is chosen sufficiently small. This implies that, for such $\tau$, $T$ is invertible.
    
    To see the claim, by the fundamental theorem of calculus for Bochner integrals applied to $g$ considered as a function of $t$, we have
    \begin{equation*}
        g(\tau,A)x=g(\tau,A)x-g(0,A)x =\int_0^\tau \frac{d}{dt}g(t,A)x \, dt = \int_0^\tau e^{tA}x \, dt.
    \end{equation*}
    Here, in order to calculate the derivative we used the series expansion of $g(t,A)$.
    
    Since $X$ is a Hilbert space, $B$ can be viewed as the analysis operator of a set of vectors $\G$. Using the properties of Bochner integrals and  H\"older inequality, we obtain
    \begin{eqnarray*}
    \|BT x\|^2&=& \sum |\langle Tx,g\rangle|^2= \sum_{g\in\G}\left|\left\langle \int_0^\tau e^{tA}x \, dt, g  \right\rangle\right|^2\\
    &=& \sum_{g\in\G}\left|\int_0^\tau\langle e^{tA}x,  g\rangle \, dt\right|^2
    \leq \sum_{g\in\G}\tau \int_0^\tau|\langle e^{tA}x,  g\rangle|^2dt\\
    &=& \tau \int_0^\tau\sum_{g\in\G}|\langle e^{tA}x,  g\rangle|^2dt
    =\tau \int_0^\tau\|Be^{tA}x\|^2dt
    \leq \tau C \|x\|^2.
    \end{eqnarray*}
\end{proof}

Approximate observability is also independent of time, this was proved by Kalman (see \cite[Theorem 1.3.1]{rdm-Niko}). As before, we enunciate it in the context of observability using Theorem \ref{rdm-connection_control_obs}.

\begin{theorem}\cite[Theorem 1.3.1]{rdm-Niko}\label{rdm-Theorem 1.3.1Niko} \index{Kalman's Theorem}\index{Theorem! Kalman's} Let $X$ and $Y$ be  Hilbert spaces, $A\in\calB(X)$ and $B\in \calB(X,Y)$. The following are equivalent. 
    \begin{enumerate}
        \item The pair $(A,B)$ associated to the continuous system \eqref{rdm-co} is AOB at some finite time $0<\tau_0<\infty$. 
        \item The pair $(A,B)$ associated to the continuous system \eqref{rdm-co} is AOB at any time $\tau\in(0,\infty)$.
        \item\label{rdm-tres} The pair $(A,B)$ associated to the continuous system \eqref{rdm-co} is AOB at infinite time.
        \item\label{rdm-cuatro} The pair $(A,B)$ associated to the discrete system \eqref{rdm-do} is AOB at infinite time. 
    \end{enumerate}
    Moreover, the equivalence between the first three items holds for general dynamic operators $A$ (not only for bounded ones).
\end{theorem}

Notice that Kalman's Theorem does not only prove the independence of time for finite time but for infinite time as well. Moreover, it also says that the property of being approximately observable is shared by both, the continuous and discrete problem, that is exactly the equivalence between \ref{rdm-tres}. and \ref{rdm-cuatro}. For approximate observability, all the possible schemes presented in this article are equivalent! 

Under the Dynamical Sampling hypotheses the previous points are exactly:
    \begin{enumerate}
    \item The family $\{e^{tA^*}g\}_{g\in\G,t\in[0,\tau_0)}$ is complete in $X$ for some $0<\tau_0<\infty$.
    \item The family $\{e^{tA^*}g\}_{g\in\G,t\in[0,\tau)}$ is complete in $X$ for any $\tau\in (0,\infty)$.
        \item The family $\{e^{tA^*}g\}_{g\in\G,t\in[0,\infty)}$ is complete in $X$.
        \item The family  $\{(A^*)^k g\}_{g\in\G,k\in\Z_+}$  is complete in $X$.
    \end{enumerate}

For exact observability, we only state a sufficient condition for time independence.  It is  done in the settings of Dynamical Sampling as an adaptation of  \cite[Theorem 5.5]{rdm-FIACPO}. 

\begin{theorem}\label{rdm-SCFrSA_rdm-FIACPO} Let  $A$ be a self-adjoint operator and $\G$ be a countable set in a Hilbert space $X$. 
	Then, $\{e^{tA^*}g\}_{g\in\G,t\in[0,1)}$ is a semi-continuous frame in $X$ if and only if $\{e^{tA^*}g\}_{g\in\G,t\in[0,\tau)}$ is a semi-continuous frame in $X$ for all finite positive $\tau$. As a consequence, the self-adjoint Dynamical Sampling problem is independent of time. 
\end{theorem}

\subsection{Infinite Time}\label{rdm-sec_inf-time}\

When considering an infinite dimensional state spaces, one of the disadvantages of finite time schemes is that we need to sample at infinite spatial points (see Corollary \ref{rdm-coro5.3_rdm-FIACPO}).
It is therefore of interest to look at infinite schemes but using only a finite number of spatial samples.

\subsubsection{Exponentially stable operators}\

In this subsection we will show that for exponentially (strongly) stable dynamic operators
it is not possible to have infinite time sampling schemes with only finitely many space samples.  

\begin{proposition}\label{rdm-Proposition6.5.2rdm-Tucsnak}  Let $X$ be a Hilbert space and let $A\in \calB(X)$ be exponentially stable (strongly stable). If the pair $(A,B)$ associated to the continuous system \eqref{rdm-co} (discrete system \eqref{rdm-do}) is exactly observable at infinite time, then it is exactly observable at some finite time. 
\end{proposition}

\begin{proof}
For the continuous system the proof can be found in \cite[Proposition 6.5.2]{rdm-Tucsnak}. Using the same arguments of that proof one can show that if the discrete system \eqref{rdm-do} is exactly observable at infinite time, then it is observable at some finite time.
Indeed, since there
exists positive constants $c_1$ and $c_2$ such that
    \begin{equation*}
        c_1\|x\|^2\leq\sum_{k\in \Z_+}\|BA^{k}x\|^2  \leq c_2\|x\|^2,
    \end{equation*}
for all $\gamma\in \Z_+$, the observability maps $\Psi_\gamma x=\sum_{k=0}^{\gamma}BA^{k}x$ are bounded operators.  Therefore we only need to see that there exists $\gamma>0$ such that the operator $\Psi_\gamma$ is bounded from below. For this, 
    \begin{eqnarray*}
        \sum_{k=0}^{\gamma}\|BA^{k}x\|^2&=&
            \sum_{k\in \Z_+}\|BA^{k}x\|^2 -     
            \sum_{k\geq \gamma+1}\|BA^{k}x\|^2\\
            &\geq& c_1\|x\|^2 - \sum_{k\in \Z_+}\|BA^{k}A^{\gamma+1}x\|^2\\
            &\geq& c_1\|x\|^2 - c_2\|A^{\gamma+1}x\|^2
            \geq \left(c_1-c_2\|A\|^{2(\gamma+1)}\right)\|x\|^2.
    \end{eqnarray*}
    Since $A$ is strongly stable, its norm is strictly less than one. Hence, we can take $\gamma\in \Z_+$ sufficiently large such that   
     $(c_1-c_2\|A\|^{2(\gamma+1)})>0$ and for such $\gamma$ we obtain that $\Psi_\gamma$ is bounded below.
\end{proof}

In the notation of Dynamical Sampling this proposition can be  rewritten as follows:
\textit{Let $\G$ be a subset of vectors in $X$. If $\ObsDynOp$ is an  exponentially stable operator  such that
	$\{e^{t\ObsAd{\ObsDynOp}}g \}_{g\in\G,t\in[0,\infty)}$ is a semi-continuous frame for $X$, then there exists $\tau<\infty$ for which $\{e^{t\ObsAd{\ObsDynOp}}g \}_{g\in\G,t\in[0,\tau)}$ is a semi-continuous frame for $X$.
	Accordingly, if $\ObsDynOp$ is a strongly stable operator such that $\{(A^*)^kg\}_{g\in\G, k\in \Z_+}$ is a frame for $X$, then there exists $\gamma\in \Z_+$ such that $\{(A^*)^kg\}_{g\in\G, n\in \{0,1,...,\gamma\}}$ is a frame for $X$.}\\

As an immediate consequence of this proposition and  Theorem \ref{rdm-Theorem 1.3.19Niko}, we obtain the following corollary.

\begin{corollary}\label{rdm-prop 6.5.2 rdm-Tucsnak} Let $A$ be an  exponentially stable (resp. strongly stable) operator in a infinite dimension Hilbert space $X$. If the pair $(A,B)$ associated to the  system \eqref{rdm-co} (resp. \eqref{rdm-do}) is EOB at infinite time, then $B$ cannot be a compact operator. In particular, it cannot be a finite rank operator.
\end{corollary}

Analogously, for the Dynamical Sampling context we recover Corollary \ref{rdm-coro5.3_rdm-FIACPO}: \textit{If $A$ is exponentially  stable (resp. strongly stable) in an infinite dimension Hilbert space and  $\{e^{tA^*}g\}_{g\in\G, t\in[0,\infty)}$ is a semi-continuous frame (resp. $\{(A^*)^kg\}_{g\in\G, k\in\Z_+}$ is a frame) for $X$, then $|\G|=\infty$.}

\subsubsection{Non-exponentially stable operators}\

For non exponentially (strongly) stable dynamic operators the scenario is quite different. We can exhibit some results that allow  to sample a signal at a finite amount of spatial points and recover it after an infinite time (continuous or discrete) evolution process.

We start with the case of discrete dynamics. For ease of the exposition, we show only the results for rank one operators $B$. Due to the Riesz Representation Theorem they are equivalent to sampling at  a single spatial point $b$.
In this case, if the dynamic operator $A$ has an orthonormal basis of eigenvectors the following theorem holds.

\begin{theorem}\label{rdm-OnePointFrame}(\cite[Theorem 3.16]{rdm-acmt})
Let $A$ be a bounded linear operator in an infinite dimensional Hilbert space $X$ with an orthogonal basis of eigenvectors $\{\varphi_n\}_{n\geq 1}$ such that $A\varphi_n = -\lambda_n \varphi_n$.
			Let $b\in X$. Then the set $\{(A^*)^kb\}_{k\in\Z_+}$ is a frame for $X$ if and only if  
		\begin{enumerate}
			\item\label{rdm-second_item}  $|\lambda_n| < 1$  for all $n\in \N.$  
			\item $|\lambda_n| \to 1$. 
			\item $\{\lambda_n\}$ satisfies Carleson's condition in the disc: 
			\begin{equation*}
			\label{rdm-carleson-cond}
			\inf_{n} \prod_{k\neq n} \left|\frac{\lambda_n-\lambda_k}{1-\overline{\lambda}_n\lambda_k}\right|\geq \delta, \, \qquad \text{ for some } \delta>0.
			\end{equation*}
			\item\label{rdm-fifth_item}  $0<c_1\le \frac {|\langle b,\varphi_n\rangle|} {\|\varphi_n\|\sqrt{1-|\lambda_n|^2}}  \le c_2< \infty$, for some constants $c_1,  c_2$ independent for all $n\in \N$.
		\end{enumerate}
	\end{theorem}
	
This theorem has its analogy in the observability question of discrete systems: \textit{Let $X$ and $A$ be as in   Theorem \ref{rdm-OnePointFrame}. Let $B$ be a rank one operator represented by $b\in X$. Then the  pair $(A,B)$ associated to the discrete system \eqref{rdm-do} is EOB at infinite time if and only if conditions 1. to 4. in Theorem \ref{rdm-OnePointFrame} hold.}

Note that this can be extended  to the case of the existence of  unconditional basis of eigenvectors, instead of an orthonormal one. This is based on the following  facts: Given $A\in \calB(X)$,  
\begin{itemize}
    \item $X$ has an unconditional basis of eigenvectors of $A$ if and only if it has a Riesz basis of eigenvectors of $A$.  Indeed, if $\{\varphi_n\}$ is unconditional basis of eigenvectors of $A$, then $\{\frac{\varphi_n}{\|\varphi_n\|}\}$ is a Riesz basis. The converse is trivial.
    \item $X$ has a Riesz basis of eigenvectors of $A$ if and only if $A$ is conjugated by an isomorphism to an operator which has an orthonormal basis of eigenvectors. In effect,  if $\{\varphi_n\}$ a Riesz basis of eigenvectors of $A$, then there exists an isomorphism $T\in \calB(X)$ such that $T\varphi_n=e_n$, where $\{e_n\}$ is an orthonormal basis of $X$ and  are the eigenvectors of the operator $\widetilde{A}:=TAT^{-1}\in\calB(X)$. The reciprocal follows similarly.
\end{itemize}

Moreover, for the cases when $A$ is a normal operator, and under discrete dynamics, it is proved in \cite{rdm-accmp} that if the set $\{(A^*)^kb\}_{k\in\Z_+}$ is a frame for $X$, then the  operator $A$ must be \textit{diagonal}.

This  can be extended to the continuous case. If the dynamic operator $A$ has an unconditional basis of eigenvectors,  from Theorem \ref{rdm-cond_nec_suf}, Proposition \ref{rdm-prop_eq_normas_eq} and Corollary \ref{rdm-prop 6.5.2 rdm-Tucsnak} we are lead to some necessary and sufficient conditions for observability of a continuous infinite-time system.

\begin{theorem}\label{rdm-conclusion1}
		Let $A$ be a bounded linear operator in an infinite dimensional Hilbert space $X$ with an unconditional basis of eigenvectors $\{\varphi_n\}_{n\geq 1}$ such that $A\varphi_n = -\lambda_n \varphi_n$. Let $B$ be a rank one operator represented by $b\in X$. Then the  pair $(A,B)$ associated to the continuous system \eqref{rdm-co} is EOB at infinite time if and only if
		\begin{enumerate}
			\item $\Real(\lambda_n) > 0$  for all $n\in\N.$  
			\item $\Real(\lambda_n)\to 0$. (It cannot be exponentially stable)
			\item \label{rdm-incond_carleson} $ \{\calE_n\}_{n\geq 1}$, given by  $\calE_n(t)=e^{-\lambda_n t}\langle b,\varphi_n \rangle$, is an unconditional sequence of $L^2([0,\infty),\mathbb{C})$.
			\item  $0<c_1\le \frac {|\langle b, \varphi_n\rangle|} {\|\varphi_n\|\sqrt{2\Real(\lambda_n)}}  \leq c_2< \infty$, for some constants $c_1,  c_2$ independent of $n$.
		\end{enumerate}
\end{theorem}

Reformulating this with the Dynamical Sampling notation one can read: \textit{Let $X$ and $A$ as in  Theorem \ref{rdm-conclusion1}. Let $b\in X$. Then, the set $\{e^{tA^*}b\}_{t\in[0,\infty)}$ is a semi-continuous frame for $X$ if and only if conditions 1. to 4. in Theorem \ref{rdm-conclusion1} hold.}

Notice that under the above hypothesis the functions $\calE_n$ are scalar-valued. By \cite[Theorem 4.2.2 and Remark 4.2.3]{rdm-Niko} we have an equivalent formulation for condition  \ref{rdm-incond_carleson}: If $\{\lambda_n\}_{n\geq 1}\subset \C$ is such that $\Real(\lambda_n)>0$ for all $n\in \N$, then  $\{e^{-\lambda_n t}\}_{n\geq 1}$ is an unconditional sequence of $L^2([0,\infty))$ if and only if the sequence $\{\lambda_n\}_{n\geq 1}$ satisfies the \textit{half-plane Carleson condition}
\begin{equation*}\label{rdm-half-plane-Carleson}
    \prod_{k\neq n} \left|\frac{\lambda_n-\lambda_k}{\lambda_n-\overline{\lambda_k}}\right|\geq \delta>0, \qquad \forall n\geq 1.
\end{equation*}

Moreover, the Carleson condition in Theorem \ref{rdm-OnePointFrame} is just an equivalent formulation for $\calE_n(k) = \langle b,\varphi_n \rangle (\lambda_n)^k$ to be an unconditional sequence of $\ell^2([0,\infty))$ (cf. \cite{rdm-acmt} and references therein).

It is also true that, if $A$ is normal, for the pair $(A,B)$ associated to the continuous system \eqref{rdm-co} to be EOB at infinite time, the operator $A$ must be diagonal. This will be presented in \cite{rdm-DDM20}.

It is worth to notice that there is not an 
equivalent formulation of Kalman's Theorem (Theorem \ref{rdm-Theorem 1.3.1Niko}) for exact observability instead of approximate observability.
Equivalently, $\{e^{tA^*}g\}_{g\in\G,t\in[0,\infty)}$ is complete if and only if $\{(A^*)^kg\}_{g\in\G,k\in\Z_+}$ is complete but it is not true that $\{e^{tA^*}g\}_{g\in\G,t\in[0,\infty)}$ is a semi-continuous frame for $X$ if and only if $\{(A^*)^kg\}_{g\in\G,k\in\Z_+}$ is a frame for $X$.

To construct a counterexample it is enough to find an operator $A$ that satisfies the conditions on Theorem \ref{rdm-OnePointFrame} (in particular condition 1.) but that fails to satisfy conditions of Theorem \ref{rdm-conclusion1}. Notice that one can start with a diagonal operator $\tilde{A}$ for this purpose and then rotate the set of eigenvalues to have a new operator $A$ that has at least one eigenvalue with $Re(\lambda_n) < 0$.

However, using the self-inverse bijection $M:\mathbb{D}\to\C_+$, given by 
\begin{equation*}\label{rdm-M del D en C}
M(z):=\frac{1-z}{1+z},    
\end{equation*}  
it easy to notice that conditions 1. to 3.  in Theorem \ref{rdm-OnePointFrame} for $\{\lambda_n\}$  are in correspondence with conditions 1. to 3. conditions in Theorem \ref{rdm-conclusion1} for the sequence $\{M(\lambda_n)\}$. 

Assuming that $|\lambda_n + 1|>\varepsilon >0$ for all $n\in \N$, and defining $\widetilde{b}\in X$ satisfying $\langle\widetilde{b},\varphi_n\rangle = \frac{\sqrt{2}}{|1+\lambda_n|}\langle b,\varphi_n\rangle$ for all $n\geq 1$, we obtain
\begin{eqnarray*}
\frac{|\langle b,\varphi_n\rangle|}{\|\varphi_n\|\sqrt{1-|\lambda_n|^2}}&=&
\frac{\sqrt{2}|\langle b,\varphi_n\rangle|}{\|\varphi_n\|\sqrt{2\Real(M(\lambda_n))} \, |1+\lambda_n|}\\
&=&
\frac{|\langle \widetilde{b},\varphi_n\rangle|}{\|\varphi_n\|\sqrt{2Re(M(\lambda_n))}}. 
\end{eqnarray*}
Therefore, condition 4. in Theorem 3.10 for $b$ is equivalent to condition 4. in Theorem 3.9 for $\widetilde{b}$. 

This yields: \textit{The pair $(A,B)$ associated to the discrete system \eqref{rdm-do} is EOB at infinite time if and only if the pair $(M(A),\widetilde{B})$ associated to the continuous system \eqref{rdm-co} is EOB at infinite time, where $\widetilde{B}$ is the analysis operator for $\widetilde{b}$.}

The condition $|\lambda_n +1|>\varepsilon>0$ is needed for $A$ and $M(A)$ to be simultaneously bounded. It can be replaced without loss of generality by $|\lambda_n - z|>\varepsilon>0$ for some $|z|=1$ if we use the map $M$ composed with a rotation of the circle.

As a final comment, it is worth noticing that, by Kalman's theorem, and the fact that EOB implies AOB, conditions 1. to 4. in Theorems \ref{rdm-OnePointFrame} and \ref{rdm-conclusion1} are sufficient conditions for AOB under any type of dynamics. That is, for example, that conditions for the discrete system to be EOB, automatically imply the continuous system to be AOB.

More general conditions for exact observability when not in the case of rank one operators can also be found in \cite[Theorem 2.3]{rdm-DSOFIS}, \cite{rdm-cms} and \cite[Theorems 4.2.4 and 4.2.5]{rdm-Niko}.

 The analogies between discrete and continuous cases will still follow the lines of the ones given in this article, although many properties turn out to be more difficult to handle than the rank one case. For details we refer to \cite{rdm-DDM20}.

\section*{Acknowledgements}

This project has mainly been carried out while R. D\'\i az Mart\'\i n hold  a postdoctoral fellowship from the CONICET, and  I. Medri hold postdoctoral fellowships from CONICET and Vanderbilt University. The authors thank those institutions for their support.

U. Molter was supported by Grants UBACyT 20020170100430BA (University of Buenos Aires),\\ PIP11220150100355 (CONICET) and PICT 2014-1480 (Ministerio de Ciencia, Tecnolog\'{\i}a e Innovaci\'on, Argentina).

\end{document}